\documentclass{amsart}
\usepackage{amsmath,amssymb,amscd,amsthm,amsxtra,array}
\usepackage[
colorlinks=true,
linkcolor=blue,
 citecolor=blue,
  urlcolor=blue,
 pagebackref,
]{hyperref}

\newtheorem{theorem}{Theorem}[section]
\newtheorem{lemma}[theorem]{Lemma}
\newtheorem{proposition}[theorem]{Proposition}

\theoremstyle{definition}

\theoremstyle{remark}

\numberwithin{equation}{section}

\newcommand{\supp}{\textup{supp}\,}
\newcommand{\absval}[1]{\mbox{$|#1|$}}
\newcommand{\strt}[1]{\rule{0pt}{#1}}
\newcommand{\norm}[1]{\mbox{$\left\| #1 \right\|$}}
\begin{document}
\title{Exponential decay estimates for Singular Integral operators}

\title[Exponential decay estimates for Singular Integral operators]{Exponential decay estimates for Singular Integral operators}

\author{Carmen Ortiz-Caraballo}
\address{Departamento de Matem\'aticas
\\ Escuela Polit\'ecnica, Universidad de Extremadura\\ Avda. Universidad, s/n, 10003 C\'aceres, Spain} \email{carortiz@unex.es}

\author{Carlos P\'erez}
\address{Departamento de An\'alisis Matem\'atico,
Facultad de Matem\'aticas, Universidad de Sevilla, 41080 Sevilla,
Spain} \email{carlosperez@us.es}

\author{Ezequiel Rela}
\address{Departamento de An\'alisis Matem\'atico,
Facultad de Matem\'aticas, Universidad de Sevilla, 41080 Sevilla,
Spain} \email{erela@us.es}
\thanks{The second author is supported by the Spanish Ministry of Science and Innovation grant MTM2009-08934,
the second and third authors are also supported by the Junta de Andaluc\'ia, grant FQM-4745.\\
}

\begin{abstract}
The following subexponential estimate for commutators is proved 
\begin{equation*}
|\{x\in Q: |[b,T]f(x)|>tM^2f(x)\}|\leq c\,e^{-\sqrt{\alpha\, t\|b\|_{BMO}}}\, |Q|, \qquad t>0.
\end{equation*}
where $c$ and $\alpha$ are  absolute constants, $T$ is a Calder\'on--Zygmund operator, $M$ is the  Hardy Littlewood maximal function  and $f$ is any function supported on the cube $Q\subset \mathbb{R}^n$. We also obtain that
\begin{equation*}
|\{x\in Q: |f(x)-m_f(Q)|>tM_{\lambda_n;Q}^\#(f)(x) \}|\le c\, e^{-\alpha\,t}|Q|,\qquad t>0,
\end{equation*}
where $m_f(Q)$ is the median value of $f$ on the cube $Q$ and $M_{\lambda_n;Q}^\#$ is Str\"omberg's local sharp maximal function with $\lambda_n=2^{-n-2}$. As a consequence  we derive  Karagulyan's estimate:   
\begin{equation*}
|\{x\in Q: |Tf(x)|> tMf(x)\}|\le c\, e^{-c\, t}\,|Q|\qquad t>0,
\end{equation*}
from \cite{K} improving Buckley's theorem \cite{Buckley}. A completely different  approach is used based on a combination of  ``Lerner's formula'' with some  special weighted estimates of Coifman-Fefferman  type obtained  via Rubio de Francia's algorithm. The method is  flexible enough to derive similar estimates for other operators such as  multilinear Calder\'on--Zygmund operators, dyadic and continuous square functions and vector valued extensions of both maximal functions and Calder\'on--Zygmund operators. In each case, $M$ will be replaced by a suitable maximal operator.



\end{abstract}
%

\maketitle

\section{Introduction}\label{intro}

A classical problem in Calder\'on--Zygmund theory is the control of a given singular operator by means  of a maximal type operator. As a model example of this phenomenon, we can take the classical Coifman--Fefferman inequality involving a Calder\'on--Zygmund (C--Z) operator and the usual Hardy--Littlewood maximal operator $M$  (see  \cite{CF}).
\begin{theorem}[Coifman--Fefferman] \label{thm:CF}
For any weight $w$ in the Muckenhoupt class $A_{\infty}$, the following norm inequality holds:
\begin{equation} \label{eq:roughC-F}
\|T^*f\|_{L^p(w)}\le c \,\|Mf\|_{L^p(w)},
\end{equation}
where $0<p<\infty$ and $c=c_{n,w,p}$ is a positive constant depending on the dimension $n$, the exponent $p$ and the weight $w$.
\end{theorem}

We use here the standard notation $T^*$ for the maximal singular integral operator of \,$T$,\, $T^*f(x)=\sup_{\varepsilon>0}|T_{\varepsilon}f(x)|$, where $T_\varepsilon$ is, as usual, the truncated singular integral. This theorem says that the ma\-xi\-mal operator $M$  plays the role of a ``control operator'' for C--Z operators, but the dependence of the constant $c$ on both $w$ and $p$ is not precise enough for some applications.  The original proof was based on the good--$\lambda$ technique introduced by Burkholder and  Gundy in \cite{BG}. The goal is to prove that the following estimate holds
\begin{equation}\label{eq:buenoslambdastw}
\left|\{x\in \mathbb{R}^n: T^*f(x)>2\lambda, Mf(x) \leq \gamma \lambda\}\right|\leq c\gamma \left|\{x\in \mathbb{R}^n: T^*(x)>\lambda\}\right|
\end{equation}
for any $\lambda>0$ and for sufficiently small $\gamma>0$. Very roughly, the main idea   to prove \eqref{eq:buenoslambdastw} is to localize  the level set $\{x\in \mathbb{R}^n: T^*f(x)>\lambda \}$  by means of Whitney cubes.  
Then the problem is reduced to study a   \emph{local} estimate of the form
\begin{equation}\label{eq:buenoslambdast}
|\{x\in Q: T^*f(x)>2\lambda, Mf(x)\leq \gamma \lambda\}|\leq c\,\gamma |Q|,
\end{equation}
where $Q$ is a cube from the Whitney decomposition and where  $f$ is supported on $Q$  and by standard methods, weighted norm inequalities for $T$ and $M$ can be derived.

In this paper we focus our attention  on the growth rate  of $\gamma$. In fact \eqref{eq:buenoslambdast} is too rough since the constant $c=c_{n, p,w}$ obtained in \eqref{eq:roughC-F} is not sharp neither on the $A_{{\infty}}$ constant  of the weight nor on $p$ as shown by Bagby and Kurtz  in \cite{BKTAMS}.

Pursuing the sharp dependece on the $A_p$ constant of the weight $w$ for the operator norm of singular integrals, Buckley  improved this good--$\lambda$ inequality \eqref{eq:buenoslambdast} (see \cite{Buckley}), obtaining a local exponential decay in $\gamma$ in the following way:
\begin{equation}\label{eq:buenoslambdastb}
|\{x\in Q: T^*f(x)>2\lambda, Mf(x)\leq \gamma \lambda\}|\leq c\,e^{-c/\gamma } |Q|.
\end{equation}
Buckley proved this estimate using as a model a more classical inequality due to Hunt for the conjugate function which was inspired by a  result of Carleson  \cite{Carleson}. 
We mention here in passing that this optimal weighted dependence, called the $A_2$ conjecture, has been proved recently and by different means by  T. Hyt\"onen in \cite{Hytonen:A2}  (see also \cite{HP}, \cite{HL-Ap-Ainf} and  \cite{HLP} for a further improvement and the recent work \cite{Lerner-simpleA2} for a very interesting simplication of the proof of the $A_2$ conjecture).  On the other hand,  this exponential decay \eqref{eq:buenoslambdastb}  has been a crucial step in deriving corresponding sharp $A_{1}$ estimate in \cite{LOP1}, \cite{LOP2}.

Our point of view is different and  has been motivated  by an improved version of  inequality \eqref{eq:buenoslambdastb} due to Karagulyan  \cite{K}:
\begin{equation}\label{eq:karagulyan}
|\{x\in Q: T^*f(x)>t Mf(x)\}|\leq c e^{-\alpha t}\, |Q|,    \qquad t>0
\end{equation}
However,  it is not clear that the  proof can be adapted to other situations.

 In the present article we present a new approach flexible enough to derive corresponding estimates for other operators. Furthermore, our approach allows to recognize and distinguish a notion of ``order of singularity'' for each operator. To be more precise and as a model, we consider a pair of operators $T_1$ an $T_2$, and consider for a fixed cube $Q$ the level set function
\begin{equation} \label{eq:generic}
\varphi(t):=\frac{1}{|Q|}|\{x\in Q: |T_1f(x)|> t|T_2f(x)|\}|, \quad t>0
\end{equation}
where $f$ stands for a function, an $m$-vector of functions or an infinite sequence, depending on the type of operators involved. In any case, all the coordinate functions are assumed to be supported on $Q$. We will provide sharp estimates on the decay rate for $\varphi(t)$ in different instances of $T_1$ and $T_2$, including the case of C--Z operators, vector-valued extensions of the maximal function or C--Z operators, commutators of singular integrals with BMO functions and higher order commutators.
 We also provide estimates for  dyadic and continuous square functions and for multilinear C--Z operators.  We summarize this different decay rates and  the maximal operators involved in Table \ref{tabla1} below (see Section \ref{sec:prelim} for the precise definitions).  
Observe that each operator has its maximal operator acting as a control operator and, further, has its specific decay rate for the corresponding level set function $\varphi(t)$.

\begin{table}[h] \label{tabla1}
\hspace{1cm}\begin{tabular}{|>{\centering} m{4.8cm}|>{\centering} m{3.2cm}| >{\centering} m{1.5cm}| m{.01cm} }
\cline{1-3}& & & \\[-.3cm]
$T_1$ & $T_2$ &   $\varphi(t)$& \\
 & & & \\ [-.3cm]
\cline{1-3}\cline{1-3}
Vector valued maximal function \\ $\overline{M}_q(\cdot)$\ , \ $1<q<\infty$& $M(|\cdot|_q)$& $e^{-\alpha t^q}$ &  \\[.6cm]
\cline{1-3}\cline{1-3}
Dyadic square function $S$ & $M$& $e^{-\alpha t^2}$ &  \\[.6cm]
\cline{1-3}\cline{1-3}
Continuous square function\\ $g_{\mu}^{*}$\, $1<\mu <\infty$ & $M$& $e^{-\alpha t^2}$ &  \\[.6cm]
\cline{1-3}
C--Z operator $T$ & $M$& $e^{-\alpha t}$ & \\[.6cm]
  \cline{1-3}
Multilinear  C--Z operator $T$ &  Multilinear \\ maximal $\mathcal{M}$& $e^{-\alpha t}$ & \tabularnewline[.6cm]
 \cline{1-3}
Vector valued extension  \\ $\overline{T}_q(\cdot)$\ , \ $1<q<\infty$ & $M(|\cdot|_q)$ &  $e^{-\alpha t}$ & \\[.6cm]
 \cline{1-3}
 Commutator $[b,T]$ & $M^2=M\circ M$& $e^{-\sqrt{\alpha t}}$ & \tabularnewline[.6cm]
  \cline{1-3}
 Iterated commutator $T^k_b$ & $M^{k+1}=M\circ\dots\circ M $\\ ($k+1$ times)& $e^{-(\alpha t)^\frac{1}{k+1}}$ & \\[.6cm]
  \cline{1-3}
\end{tabular}

\vspace{.8cm}

 \caption{Order of singularity for several operators}
\end{table}

In this work we will present two different approaches, both based on the use of  Lerner's formula (see Theorem \ref{thm:andreiformula}), which is a very powerful and succesful method as we can see in several recent situations (see \cite{CMP-ERAMS}, \cite{CMP-ADV}). Roughly, the first approach allows us to derive the exponential decay whenever there is a superlinear rate, namely, in all the cases except for the commutators. This method, although it is far from being trivial, can be seen as the natural way to exploit Lerner's formula to obtain the exponential decay. However, it fails when we consider the case of commutators.  Hence, to be able to tackle this latter case, we develop a different method, which is the more original and substantial contribution of the present article: a novel approach and a different type of proof based on weighted estimates. 
This second approach uses Lerner's formula to derive suitable local versions of weighted norm inequalities of Coifman--Fefferman type. This, combined with factorization arguments, gives \emph{all} the results, including commutators of any order. In addition, we present here a sort of ``template'', a general scheme that can be applied to any pair of operators fulfilling certain general hypothesis.

The paper is organized as follows. In Section \ref{sec:statements} we present the precise statement of our results. In Section \ref{sec:prelim} we include some preliminary definitions and tools needed in the sequel. In Section \ref{sec:Proofs-first} we present our first approach and provide the proofs of the ``superlinear'' results. In Section \ref{sec:Proofs-second} 
we present our second approach and prove \emph{all} the results of the paper. In this final section we also include some background on weights and, in addition, some new extensions of classical results.

\section{Statement of the main results}\label{sec:statements}

In this section we present the precise statement of the main results of this paper. We start with a general result involving a generic a function $f$ and its \emph{local maximal function} $M_{\lambda;Q}^\#f$ in a given cube $Q$ (see Section \ref{sec:prelim} for the precise definitions).

\

\textbf{$\bullet$ The key estimate: a John-Str\"omberg-Fefferman--Stein type inequality}

\begin{theorem}\label{thm:Feff-Steinish}
Let $Q$ be a cube and let $f\in L_{c}^\infty(\mathbb{R}^n)$  such that $\emph{supp}(f)\subseteq Q$. Then there are constants  $\alpha,\ c>0$ such that
\begin{equation}\label{eq:karagulyan-feff-stein}
|\{x\in Q: |f(x)-m_f(Q)|>tM_{2^{-n-2};Q}^\#(f)(x) \}|\le c e^{-\alpha t}|Q|,\quad t>0.
\end{equation}
\end{theorem}

Such an estimate involving a function controlled in some sense by its sharp maximal function is surely related to Fefferman--Stein inequality, but the version we present here with the \emph{local} sharp maximal function goes back to the work of Str\"omberg \cite{Stromberg79} and Jawerth and Torchinsky \cite{JawTor85}.

Once we have such a general theorem, we can derive the results announced in the introduction for a wide class of singular operators. The idea is to apply the theorem to a given singular operator $\mathcal{T}$ and then use the key tool: a pointwise estimate of the form $M_{2^{-n-2};Q}^\#(\mathcal{T}f)(x)\le c\, \mathfrak{M}(f)(x)$,
where $\mathfrak{M}$ is an appropriate maximal operator.

More precisely, we have the following theorems:

\

$\bullet$ Calder\'on-Zygmund Operators.

\begin{theorem}\label{thm:singinteg}
Let $T$ be a C-Z operator  with maximal singular integral operator $T^*$. Let $Q$ be a cube and let $f\in L_{c}^\infty(\mathbb{R}^n)$  such that $\emph{supp}(f)\subseteq Q$. Then there are constants  $\alpha, c >0$ such that
\begin{equation}\label{eq:karasi}
|\{x\in Q: |T^*f(x)|>tMf(x) \}|\le c e^{-\alpha t}|Q|,\qquad t>0.
\end{equation}
\end{theorem}

\

$\bullet$ Calder\'on-Zygmund Multilinear Operators.

\begin{theorem}\label{thm:multisinginteg}
Let $T$ be a $m$-linear C-Z operator. Let $Q$ be a cube and let $\vec{f}\,$ be vector of $m$ functions $f_j\in L_{c}^\infty(\mathbb{R}^n)$  such that $\emph{supp}(f_j)\subseteq Q$ for $1\le j\le m$. Then there are constants $\alpha, \ c>0$ such that
\begin{equation}\label{eq:karasimulti}
|\{x\in Q: |T\vec{f}\,(x)|>t\mathcal{M}\vec{f}\,(x) \}|\le c e^{-\alpha t}|Q|,\qquad t>0.
\end{equation}
\end{theorem}

\

$\bullet$ Vector-valued extensions.

\begin{theorem}\label{thm:vectorvaluedSingular}
Let $1<q<\infty$ and let $\overline{T}_{q}$ be the vector-valued extension  of $T$, where $T$ is a C-Z operator. Then there are constants $\alpha, \ c>0$ such that for any cube $Q$ and any vector-function $f=\{f_j\}_{j=1}^{\infty}$ with  $\supp f\subseteq Q$:
\begin{equation} \label{eq:karasivec}
|\{x\in Q: \overline{T}_{q}f(x)> tM(|f|_q)(x)\}|\leq c e^{-\alpha t}\, |Q|,\qquad t>0.
\end{equation}

\end{theorem}

\

\begin{theorem}\label{thm:vectorvaluedMaximal}
Let $1<q<\infty$ and let\, $\overline{M}_{q}$ be the vector-valued extension of $M$. Then there are constants $\alpha, \ c>0$ such that for any cube $Q$ and any vector-function $f=\{f_j\}_{j=1}^{\infty}$ with  $\supp f\subseteq Q$:
\begin{equation} \label{eq:karasivecmaximal}
|\{x\in Q: \overline{M}_{q}f(x)> tM(|f|_q)(x)\}|\leq c e^{-\alpha t^q}\, |Q|,\qquad t>0.
\end{equation}
\end{theorem}

\

$\bullet$ Littlewood-Paley square functions.

\begin{theorem}\label{thm:squarefunction}
Let  $S$ be the dyadic square function and let $g_{\mu}^*$ be the continuous Littlewood-Paley square function, with $\mu>3$. Let $Q$ be a cube and let $f\in L_{c}^\infty(\mathbb{R}^n)$  such that $\emph{supp}(f)\subseteq Q$. Then there are constants $\alpha, \ c>0$ such that
\begin{equation}\label{eq:karasquare}
|\{x\in Q: Sf(x)>tMf(x) \}|\le c e^{-\alpha t^2}|Q|,\qquad t>0.
\end{equation}
and
\begin{equation}\label{eq:g^*square}
|\{x\in Q: g_{\mu}^*(f)(x)>tMf(x) \}|\le c e^{-\alpha t^2}|Q|,\qquad t>0.
\end{equation}

\end{theorem}

\

We also present here the result for commutators, although it will not follow from Theorem \ref{thm:Feff-Steinish}. We will prove this theorem following the ``weighted approach'' announced in the introduction.

\

$\bullet$ Commutators.

\begin{theorem}\label{thm:commutator}
Let $T$ be a Calder\'{o}n--Zygmund operator an let $b$ be in $BMO$. Let $f$ be a function such that $\supp f\subseteq Q$. Then there are constants, such that
\begin{equation} \label{eq:karasicommu}
|\{x\in Q: |[b,T]f(x)|>tM^2f(x)\}|\leq c e^{-\sqrt{\alpha t\|b\|_{BMO}}}\, |Q|, \qquad t>0.
\end{equation}

Similarly,  for  higher commutators we have 
\begin{equation} \label{eq:karasimcommu}
|\{x\in Q: |T_b^kf(x)|> tM^{k+1}f(x)\}|\leq c e^{-(\alpha t\|b\|_{BMO})^{1/(k+1)}}\, |Q|,
\end{equation}
for all  $t>0$.
\end{theorem}

\section{Preliminaries and notation}\label{sec:prelim}

In this section we gather  some well known  definitions and properties which will be used along this paper. We will adopt the usual notation $f_Q= \frac{1}{|Q|}\int_Q f(y)\, dy.$ for the average over a cube $Q$ of a function $f$.

\subsection{Maximal Functions}
Given a locally integrable function $f$ on $\mathbb{R}^n$, the \emph{Hardy--Littlewood maximal operator} $M$ is defined by
\begin{equation*}
Mf(x)=\sup_{Q\ni x}{\frac1{|Q|}}\int_{Q} f(y)\\,dy,
\end{equation*}
where the supremum is taken over all cubes $Q$ containing the point $x$.  For $\varepsilon>0$, we define:
$$
M_\varepsilon  f(x)=(M (|f|^\varepsilon)(x))^{1/\varepsilon}.
$$
The  usual sharp maximal function of Fefferman--Stein is defined as:
$$M^\#(f)(x)=\sup_{Q\ni x}\inf_c\frac{1}{|Q|}\int_Q |f(y)-c|\, dy,  $$
We will also use the following operator:
$$M^\#_\delta(f)(x)=\sup_{Q\ni x}\inf_c\left(\frac{1}{|Q|}\int_Q |f(y)-c|^\delta\, dy,\right)^\frac{1}{\delta}.  $$
If the supremum is restricted to the dyadic cubes, we will use respectively the following notation  $M^d$, $M_\delta ^{\#,d}$ and $M_\delta ^{d}$.
We will also need to consider iterations of maximal functions. Let $M^k$ be defined as
$$M^k:= M\circ \cdot\cdot  \circ M \qquad (k\, times).$$
In addition, for a given cube $Q$, we will consider local maximal functions. For a fixed cube $Q$, we will denote by $\mathcal{D}(Q)$ to the family of all dyadic subcubes with respect to the cube $Q$. The maximal function $M^Q$ is defined by
\begin{equation*}
M^Qf(x)=\sup_{P\in \mathcal{D}(Q), P\ni x}{\frac1{|P|}}\int_{P} f(y)\\,dy.
\end{equation*}
Similarly, $M_\delta^Q$, $M^{\#,Q}$ and $M^{\#,Q}_\delta$ are defined in the same way as above.

We introduce the following notation: for a given vector--valued function ${f}=(f_j)_{j=1}^{\infty}$ we denote
$$
|f(x)|_q:= \left(\sum_{j=1}^{\infty} |f_j(x)|^q \right)^{1/q}.
$$ 
Then, the classical \emph{vector-valued extension of the maximal function} introduced by Fefferman and Stein in \cite{FS} can be written as follows:
$$\overline{M}_qf(x)=\Big( \sum_{j=1}^{\infty} (Mf_j(x))^q \Big)^{1/q}=|Mf(x)|_q,$$
where ${f}=\{f_j\}_{j=1}^{\infty}$ is a vector--valued function.

Within the multilinear setting, the appropriate maximal function $\mathcal{M}$ for a $m$-vector $\vec{f}\,$ of $m$ functions $\vec{f}\,=(f_{1},\dots,f_{m})$ is defined as
\begin{equation}\label{eq:maxmultilinear}
\mathcal{M}(\vec{f}\,)(x)=\sup_{Q\ni x} \prod_{i=1}^m\frac{1}{|Q|}\int_{Q}|f_{i}(y_{i})|\ dy_{i}.
\end{equation}
Note that this operator is pointwise smaller than the $m$-fold product of $M$. This maximal operator was introduced in in \cite{LOPTTG} where it is shown that is the ``correct'' maximal operator controlling the multilinear C--Z operators.

\subsection{Calder\'on--Zygmund operators}
We will use standard well known definitions, see for instance \cite{Journe,GrafakosCMF}. Let $K(x,y)$ be a locally integrable function defined of the diagonal $x=y$ in $\mathbb{R}^n\times \mathbb{R}^n$, which satisfies the size estimate
\begin{equation}\label{eq:tamano}
|K(x,y)|\leq {\frac{c}{|x-y|^n}},
\end{equation}
and for some $\varepsilon>0$, the regularity condition
\begin{equation}\label{eq:regularidad}
|K(x,y)-K(z,y)|+|K(y,x)-K(y,z)|\leq
c{\frac{|x-z|^{\varepsilon}}{|x-y|^{n+\varepsilon}}},
\end{equation}
whenever $2|x-z|<|x-y|$.

A linear operator $T:C_{c}^{\infty}(\mathbb{R}^n)\longrightarrow L_{loc}^{1}(\mathbb{R}^n)$ is a \emph{Calder\'{o}n--Zygmund operator} if it extends
to a bounded operator on $L^2(\mathbb{R}^n)$, and there is a kernel $K$ satisfying \eqref{eq:tamano} and \eqref{eq:regularidad} such that
\begin{equation}
Tf(x)=\int_{\mathbb{R}^n}K(x,y)f(y)\, dy,
\end{equation}
for any $f\in C_{c}^{\infty}(\mathbb{R}^n)$ and $x\notin supp(f)$.

Given a C--Z operator $T$ we define as usual the vector--valued extension $\overline{T}_{q}$ as
\[
\overline{T}_{q}f(x)=
\left(
\sum_{j=1}^{\infty} \absval{ Tf_{j}(x) }^{q}
\right)^{1/q}=|Tf(x)|_q,
\]
where ${f}=\{f_j\}_{j=1}^{\infty}$ is a vector--valued function.

We will also study the problem in the \emph{multilinear} setting, considering multilinear C--Z operators acting on pro\-duct Lebesgue spaces. Let $T$ be an operator initially defined on th $m$-fold product of Schwartz spaces and taking values into the space of tempered distributions,
\begin{equation*}
T:\mathcal{S}(\mathbb{R}^n)\times\dots\times\mathcal{S}(\mathbb{R}^n)\to \mathcal{S}'(\mathbb{R}^n).
\end{equation*}
We say that $T$ is an $m$-linear C--Z operator if, for some $1\le q_j <\infty$, it extends to a bounded multilinear operator from $L^{q_1}\times\dots\times L^{q_m}$ to $L^q$, where $\frac{1}{q}=\frac{1}{q_1}+\dots + \frac{1}{q_m}$ and if there exists a function $K$ defined off the diagonal $x=y_1=\dots=y_m$ in $(\mathbb{R}^n)^{m+1}$, satisfying
\begin{equation*}
 T(f_1,\dots,f_m)(x)=\int_{(\mathbb{R}^n)^m}K(x,y_1,\dots,y_m)f_1(y_1)\dots f_m(y_m)\ dy_1\dots dy_m
\end{equation*}
for all $x\notin\bigcap_{j=1}^m \text{supp}f_j$. We refer to   \cite{GT} and \cite{LOPTTG} for a detailed treatment of these operators.

\subsection{Commutators}
Let $T$ be any operator and let $b$ be any locally integrable function. The \emph{commutator operator} $[b, T]$ is defined by
$$
[b,T]f=b\,T(f)-T(bf).
$$
If  $b\in BMO$ and $T$ is a C-Z operators these operators were considered by Coifman, Rochberg and Weiss.  These operators  are  more singular than a C--Z operator, a fact that can be seen from  the following version of the classical result of Coifman and Fefferman \eqref{eq:roughC-F} for commutators proved by the second author in \cite{perez97}.   One of the main points of this paper is  that there is an intimate connection between these commutators and iterations of the Hardy-Littlewood maximal operator.

An important point is that these operators are not of weak type $(1,1)$, but we do have the following substitute inequality.
\begin{theorem}\cite{Perez95:JFA} \label{thm:weakcomm}
Let $b$ be a $BMO$ function and let $T$ be a C-Z operator. Defined the function $\phi(t)$ as follows $\phi(t)= t(1+\log^{+}t)$, there exists a positive constant $c=c_{\|b\|_{BMO}}$ such that for all compactly supported function $f$ and for all $\lambda>0$,
$$
\left|\{x\in \mathbb{R}^n: |[b,T]f(x)|>\lambda\}\right| \leq c\, \int_{\mathbb{R}^n}\phi\left(\dfrac{|f(x)|}{\lambda}\right)\,dx.
$$
\end{theorem}
A natural generalization of the commutator $[b,T]$ is given by $T^k_{b}:=[b,T_{b}^{k-1}]$,\,   $k\in\mathbb{N}$ and more explicitly by,
\begin{equation*}
T^k_b f(x)=\int_{\mathbb R^n} (b(x)-b(y))^k K(x,y)f(y)\,dy.
\end{equation*}
We call them \emph{higher order commutators} and the case $k=0$ recaptures the Calder\'on--Zygmund singular integral operator, and for $k=1$ we get the commutator operator defined before.  It is shown in  \cite{perez97} that for any $0<p<\infty$ and any $w\in A_{\infty}$ there is a constant $C$ such that
%
%
Again, this inequality is sharp since, $M^{k+1}$ can not be replaced by the smaller operator $M^k$.

\subsection{Littlewood-Paley square functions. }

Let $\mathcal{D}$ denote the collection of dyadic cubes in $\mathbb{R}^n$.
Given $Q\in \mathcal{D}$, let $\widehat{Q}$ be its dyadic parent, i.e.,  the
unique dyadic cube containing $Q$  such that $|\widehat{Q}|=2^n |Q|$. The dyadic square function is the operator
$$S_df(x) = \left(\sum_{Q\in\mathcal{D}}(f_Q-f_{\widehat{Q}})^2\chi_Q(x)\right)^{1/2},$$
where as usual $f_Q$ denotes the average of $f$ over $Q$.
For the properties of the dyadic square function we refer the reader
to Wilson~\cite{wilson-LNM}.

We will also use the following continuous and more classical version of the square function:
\begin{equation}\label{eq:LPsquare}
g_{\strt{1.7ex}\lambda}^*(f)(x) = \left( \int_{0}^{\infty}
\int_{\mathbb{R}^n} |\phi_t*f(y)| ^2
\left(\frac{t}{t+|x-y|}\right)^{n\lambda} \frac{dy\,dt}{t^{n+1}}
\right)^{1/2},
\end{equation}
where $\phi\in\mathcal S$, $\int\phi\,dx=0$, $\phi_t(x)=\frac{1}{t^n}\phi(\frac{x}{t})$, and $\lambda>2$ (see \cite{Stein-PrincetonBook}).

\subsection{Lerner's formula}In this subsection, we will state a result from \cite{Lerner:formula} which will be fundamental in our proofs.  This result is known as ``Lerner's formula'', and allows to obtain a decomposition of a function $f$ that can be seen as a sophisticated  Calder\'on--Zygmund decomposition of that function at all scales.

In order to state Lerner's result, we need to introduce the main objects involved. For a 
given a cube  $Q$, the median value $m_f(Q)$ of $f$ over $Q$ is a, possibly non-unique, number such that
$$|\{x\in Q:f(x)>m_f(Q)\}|\le |Q|/2$$
and
$$|\{x\in Q:f(x)<m_f(Q)\}|\le |Q|/2.$$
The mean local oscillation of a measurable function $f$ on a cube $Q$ is defined by the following expression
$$
\omega_{\lambda}(f;Q)=\inf_{c\in \mathbb{R}}((f-c)\chi_{Q})^{*}(\lambda |Q|),
$$
for all $0<\lambda<1$, and the local sharp maximal function on a fixed cube $Q_0$ is defined as
$$
M^{\#}_{\lambda;Q_0}f(x)=\sup_{x\in Q\subset Q_0}\omega_{\lambda}(f;Q),
$$
where the supremum is taken over all cubes $Q$ contained in $Q_0$ and such that $x\in Q$. Here $f^*$ stands for the usual non-increasing rearrangement of $f$. We will use several times that for any $\delta>0$, and $0<\lambda \le 1$,
\begin{equation}\label{eq:rearravg-delta}
(f\chi_{\strt{1.7ex}Q })^*(\lambda |Q|)\le \left(\frac{1}{\lambda|Q|}\int_Q|f|^{\delta}\,dx\right)^{1/{\delta}},
\end{equation}
and, as a consequence, that
\begin{equation}\label{eq:medianVSnormdelta}
|m_f(Q)|\le \left(\frac{2}{|Q|}\int_Q|f(x)|^{\delta}\,dx\right)^{1/\delta},
\end{equation}
for any $\delta>0$.

Recall that, for a fixed cube $Q_0$, $\mathcal{D}(Q_0)$ denotes all the dyadic subcubes with respect to the cube $Q_0$. As before, if $Q\in \mathcal{D}(Q_0)$ and $Q\neq Q_0$, $\widehat{Q}$ will be the ancestor dyadic cube of $Q$, i.e., the only cube in $\mathcal{D}(Q_0)$ that contains $Q$ and such that $|\widehat{Q}|=2^n |Q|$. We state  now  Lerner's formula. 
\begin{theorem}\cite{Lerner:formula}\label{thm:andreiformula} Let f be a measurable function on $\mathbb{R}^n$ and let $Q_0$ be a cube. Then
there exists a (possibly empty) collection of cubes $\{Q_j^k\}_{j,k}\in \mathcal{D}(Q_0)$ such that:
\begin{enumerate}
  \item[(i)] For a.e. $x\in Q_0$,
   \begin{equation}\label{eq:aformula}
   |f(x)-m_f(Q_0)|\leq 4\,M_{1/4; Q_0}^{\#}f(x)+4\,\sum_{k=1}^{\infty}\sum_{j}\omega_{1/2^{n+2}}(f;\hat{Q}_j^k)\chi_{Q_j^k}(x);
   \end{equation}
  \item[(ii)] For each fixed $k$ the cubes $Q_j^k$ are pairwise disjoint;
  \item[(iii)] If $\Omega_k={\displaystyle{\bigcup_{j}}}Q_j^k$, then $\Omega_{k+1}\subset\Omega_k$;
  \item[(iv)] $|\Omega_{k+1}\cap Q_j^k|\leq {\frac1{2}}|Q_j^k|$.
\end{enumerate}
\end{theorem}
Let us remark that in any decomposition as in the previous theorem, if we define 
$E_j^k:=Q_j^k\backslash\Omega_{k+1}$, then we have that $\{E_j^k\}$ is a pairwise disjoint subsets family. Moreover, 
  \begin{equation}\label{eq:remark-Lerner}
   |Q_j^k|\leq 2|E_j^k|.
  \end{equation}

\subsection{Pointwise inequalities}

In this section we will summarize some important pointwise inequalities involving sharp maximal functions. We start with the following, which is an immediate consecuence of the definitions. Given a cube $Q$, $\delta>0$ and $0<\lambda\le 1$, there exists a constant $c=c_\lambda$ such that
\begin{equation}\label{eq:localsharp-VS-sharp}
M_{\lambda;Q}^{\#}(f\chi_Q)(x)\leq c\, M_\delta^{\#} (f\chi_Q)(x),
\end{equation}
for all $x\in Q$.
We will also use the following result from \cite{O}.
If $0<\delta<\varepsilon<1$, there is a constant $c=c_{\varepsilon,\delta}$ such that
\begin{equation}\label{eq:dosmaximales}
M^{\#,d}_{\delta}(M^d_{\varepsilon}(f))(x)\le c \,M^{\#,d}_{\varepsilon}f(x).
\end{equation}

The idea behind the following list of inequalities is that a sharp maximal type operator acting on several singular operators can be controlled by suitable maximal operators. 

{\it Calder\'on--Zygmund operators and vector valued extensions:}
Let $T$ be a Calder\'on--Zygmund operator  with maximal singular operator $T^*$, and $0<\varepsilon<1$. Then there exists a constant $c=c_{\varepsilon}$ such that
\begin{equation}\label{eq:MsharpT-vs-Mf}
M_{\varepsilon}^{\#}(T^*f)(x)\leq c\,Mf(x).
\end{equation}
This follows essentially from \cite{AP94} where $T$ is used instead of $T^*$. Moreover, we know from 
\cite{PTG} that if 
$1<q<\infty$ and $0<\varepsilon<1$, then there
exists a constant $c=c_{\varepsilon}>0$ such that
\begin{equation}\label{eq:vectorialmaximalsharp}
M_{\varepsilon}^{\#}(\overline T_qf)(x)\leq c\,M\left(|f|_q\right)(x) \quad x\in\mathbb{R}^n
\end{equation}
for any smooth vector function $f=\{f_j\}_{j=1}^{\infty}$.

{\it Multilinear C--Z operators.} \cite{LOPTTG}
Let $T$ be a  Calder\'on-Zygmund $m$-linear operator and let $0<\varepsilon<1/m$.  Then there exists a constant $c=c_{\varepsilon}>0$ such that
\begin{equation}\label{eq:multilinmaximalsharp}
M^\#_{\varepsilon}(T(\vec{f}\,))(x)\le c\,\mathcal{M}(\vec{f}\,)(x)\quad x\in\mathbb{R}^n
\end{equation}
for any smooth vector function $\vec{f}$.	

{\it Commutators.} \cite{Perez95:JFA}
Let $b\in BMO$  and let $0 < \delta< \varepsilon$. Then there exists a positive constant $c= c_{\delta,\varepsilon}$ such that,
\begin{equation}\label{eq:mdeltaepsilon}
M_{\delta}^{\#,d}(T^k_bf)(x)\leq c\, \norm
{b}_{BMO}\sum_{j=0}^{k-1}M^d_{\varepsilon}(T^j_b f)(x)+\norm{b}_{BMO}^k M^{k+1}f(x),
\end{equation}
for any $k\in \mathbb{N}$ and for all smooth functions $f$.

{\it Dyadic and continuous square functions.}\cite{CMP-ADV},\cite{Lerner-JFA}
Let $S_d$ be the dyadic square function operator and let  $0<\lambda<1$. Then for any function $f$, every dyadic cube $Q$, and every $x\in Q$,
\begin{equation}\label{eq:squareOscilaDyadic}
\omega_\lambda((S_df)^2,Q) \le \frac{c_n}{\lambda^2}\left(\frac{1}{|Q|}\int_Q |f(x)|\ dx\right)^2,
\end{equation}
and hence
\begin{equation}\label{eq:sharplocal-SquareDyadic}
 M^{\#,d}_{\lambda}(S_{d}(f)^2)  (x) \leq c_{\lambda}\,Mf(x)^2.
\end{equation}
For the continuous square function   $g_{\mu}^{*}$ we use the following from \cite{Lerner-JFA}. 
For $\mu>3$ and $0<\lambda<1$, we have that
\begin{equation}\label{eq:sharplocal-SquareCont}
 M^{\#}_{\lambda}(g_{\mu}^{*}(f)^2)  (x) \leq c_{\lambda}\,Mf(x)^2.
\end{equation}

The analogue for the vector-valued extension of the maximal function, from \cite{CMP-ADV} is the following.
Fix $\lambda$, $0<\lambda<1$ and $1<q<\infty$. Then for any function $f=\{f_j\}_{j=1}^{\infty}$, every dyadic cube $Q$, and every $x\in Q$,
\begin{equation}\label{eq:vectormaximalOscila}
\omega_\lambda\left(\left(\overline{M}^d_qf\right)^q,Q\right) \le \frac{c_{n,q}}{\lambda^q}\left(\frac{1}{|Q|}\int_{Q} \norm{f(x)}_{l^q}\ dx\right)^q
\end{equation}

Finally, we include here the well known Kolmogorov's inequality in the following form.  Let $0<q<p<\infty$. Then there is a constant $c=c_{p,q}$ such that for any nonegative measurable function $f$,
\begin{equation}\label{eq:kolmogorov}
\left(\frac{1}{|Q|}\int_Q f(x)^q\ dx\right)^{\frac{1}{q}}\le c \|f\|_{L^{p,\infty}(Q,\frac{dx}{|Q|})}.
\end{equation}
(See for instance  \cite{GrafakosCF}, p. 91, ex. 2.1.5).

\section{First approach, proof of linear and superlinear estimates}\label{sec:Proofs-first}

We prove in this section Theorem \ref{thm:Feff-Steinish} and the  consequences. The proof is based on Lerner's formula \eqref{eq:aformula} combined with a new way of handling the sparse cubes $\{Q_j^k\}$ by means of an exponential vector valued endpoint estimate due to Fefferman--Stein.

\subsection{Proof of the key estimate}
As already mentioned the proof is based on Lerner's formula from Theorem  \ref{thm:andreiformula}. The drawback of the method is that it is not clear if this approach allows us to derive such sharp exponential decays for the case of commutators. We remark that a slightly weaker result, involving $M^\#_\delta$ instead of the local sharp maximal function was proved by the second author in \cite{CGPSZ}, Chapter 3.

\begin{proof}[of Theorem \ref{thm:Feff-Steinish}]
We consider the distribution set
$$E_Q:=\{x\in Q: |f(x)-m_f(Q)|>tM_{2^{-n-2};Q}^\#(f)(x) \}.$$
Then, by \eqref{eq:aformula} and for appropriate $c$ we have that 
\begin{eqnarray*}
\left|E_Q\right|&\le&|\{x\in Q:  \sum_{k,j}\chi_{Q_j^k}(x)\, \inf_{Q_j^k}M_{2^{-n-2};Q}^\#f  > ctM_{2^{-n-2};Q}^\#f(x)\}|\\
&\leq& |\{x\in Q:  \sum_{k,j}\chi_{Q_j^k}(x)> ct\}|.
\end{eqnarray*}
Let $\{E_j^k\}$ be the family of sets from the remark after Lerner's formula satisfying \eqref{eq:remark-Lerner}. We have then
\begin{eqnarray*}
\sum_{j,k}\chi_{Q_j^k}(x)&=&\sum_{j,k}\left({\frac{1}{|Q_j^k|}}\,|Q_j^k|\right)^q\chi_{Q_j^k}(x)\\
&\leq& c_n^q\,\sum_{j,k}\left({\frac{1}{|Q_j^k|}}\,|E_j^k|\right)^q\chi_{Q_j^k}(x)\\
&\leq& c_n^q\,\sum_{j,k}\left({\frac{1}{|Q_j^k|}}\,\int_{Q_j^k}\chi_{E_j^k}(x)\,dx\right)^q\chi_{Q_j^k}(x)\\
&\leq& c_n^q\, \left(\overline{M}_q\left(\left\{\chi_{E_j^k}\right\}_{j,k}\right)(x)\right)^q\\
&\leq& c_n^q\, \left(\overline{M}_qg(x)\right)^q,
\end{eqnarray*}
where $g=\left\{\chi_{E_j^k}\right\}_{j,k}$.  Now, since $\{E_j^k\}$ is a pairwise disjoint family of subsets, we have that
\begin{equation}
\|g(x)\|_{\ell^{q}}=\left(\sum_{j,k}\left(\chi_{E_j^k}(x)\right)^q\right)^{1/q}\leq 1,
\end{equation}
We finish our proof recalling that if $|g|_{\ell^{q}}\in L^{\infty}$, then $\left(\overline{M}_qg(x)\right)^q\in Exp L$ (see \cite{FS}). Therefore, we obtain the desired inequality \eqref{eq:karagulyan-feff-stein}:
\begin{equation*}
|\{x\in Q: |f(x)-m_f(Q)|>tM_{2^{-n-2};Q}^\#(f)(x) \}|\le c e^{-\alpha t}|Q|,\qquad t>0.
\end{equation*}
\end{proof}

\subsection{Proofs for Calder\'on--Zygmund operators, vector valued extensions and  multilinear C--Z operators - First approach}

We will combine  Theorem \ref{thm:Feff-Steinish}, replacing $f$ by the operator, with an appropriate pointwise inequality. We start by proving  Theorem \ref{thm:singinteg}.

\begin{proof}[of Theorem \ref{thm:singinteg}]
We first note the following estimate for the median value of $T^*f$ over a cube $Q$. We have that
\begin{eqnarray*}
m_{T^*f}(Q) & \leq &  \left(\frac{2}{|Q|}\int_{Q}(T^*f)^{\delta}\right)^{1/\delta}\\
& \le & c_{\delta}\|T^*f\|_{L^{1,\infty}(Q,\frac{dx}{|Q|})}\leq {\frac{c}{|Q|}}\int_{Q}|f(x)|\,dx\\
\end{eqnarray*}
by Kolmogorov's inequality \eqref{eq:kolmogorov}.
It follows that
\begin{equation}\label{eq:medianT-vs-M}
m_{T^*f}(Q) \le  cMf(x), \qquad x\in Q.
\end{equation}
This, together with inequality \eqref{eq:localsharp-VS-sharp} and \eqref{eq:MsharpT-vs-Mf}, yields
\begin{equation*}
\left|\left\{x\in Q: \frac{|T^*f(x)|}{Mf(x)}>t \right\}\right|\le |\{x\in Q: |T^*f(x)|>ctM_{\lambda_n,Q}^\#(T^*(f))(x) \}|
\end{equation*}
for $\lambda_n=2^{-n-2}$ for some constant $c>0$. We can apply now our general result from Theorem \ref{thm:Feff-Steinish} to conclude the proof.

\end{proof}

For the proof of Theorem \ref{thm:multisinginteg} and Theorem \ref{thm:vectorvaluedSingular}, we have all the ingredients: we use (respectively) inequalities \eqref{eq:multilinmaximalsharp} and \eqref{eq:vectorialmaximalsharp} instead of \eqref{eq:MsharpT-vs-Mf} and we control the median value by using Kolmogorov's inequality and the weak type of both vector-valued extensions and multilinear C--Z operators.

\subsection{Proof for the square functions and for the vector-valued maximal function - First approach}

For the proof of Theorem \ref{thm:squarefunction} we  start with:
\begin{equation*}
|\{x\in Q: Sf(x)>tMf(x) \}|=|\{x\in Q: (Sf(x))^2>t^2(Mf(x))^2 \}|.
\end{equation*}
and we use this time  estimates \eqref{eq:sharplocal-SquareDyadic} and \eqref{eq:sharplocal-SquareCont} for the pointwise control. The median value of the square function is also bounded by $M$ as in the previous cases using, again, Kolmogorov's inequality and the weak $(1,1)$ type of the operator. From Theorem \ref{thm:Feff-Steinish} we will obtain, in this case, a Gaussian decay rate for the level set. 

Finally, for the vector-valued extension of the Maximal function, we proceed as in the case of the square function but replacing the ``2'' by ``$q$''. The key estimate for the oscillation is in inequality \eqref{eq:vectormaximalOscila}.

\section{Second approach - Weighted estimates and the proof for commutators}\label{sec:Proofs-second}

As already mentioned the approach considered in the previous section cannot be used in the case of commutators.  We introduce here a new approach, combining  Lerner's formula with a variant of Rubio de Francia's algorithm. In this case, Lerner's formula is used to derive a certain sharp local weighted estimate (see Theorem \ref{thm:CFlineal}). This is the first key ingredient. The second key ingredient is to apply Rubio de Francia's algorithm with a factorization argument for $A_q$ weights and the use of Coifman--Rochberg theorem (see lemma \ref{lem:multi-CR}).

This approach will allow us to derive all the results of this paper, including those proved in the previous section, and also the results for commutators. We will present the general scheme in terms of a pair of generic operators $T_1$ and $T_2$ and then emphasize the different kind of hypothesis needed and the estimates obtained on each case.

We start with some preliminaries about weights. We include some classical well known results and some new ones.

\subsection{Some extra preliminary on weights.}
We recall that a weight $w$ (any non negative measurable function) satisfies the $A_p$ condition for $1<p<\infty$ if
\begin{equation*}
[w]_{A_p}=  \left( {\displaystyle{\frac1{|Q|}}}\int_{Q}w \right)
\left({\displaystyle{\frac1{|Q|}}}\int_{Q}w
^{1-p'}\right)^{p-1}<\infty.
\end{equation*}
Also we recall that $w$ is an $A_1$ weight if there is a finite constant $c$ such that $Mw\le c\,w$ a.e., and where $[w]_{A_1}$ denotes the smallest of these $c$. Also, we recall that the $A_{\infty}$ class of weights is defined by  $A_{\infty}=\bigcup_{{p\geq 1}}A_p$.

We will use that if $w_1$\, and \,$w_2$\, are $A_1$  weights then $w=w_1w_2^{1-p} \in A_p$
and
\begin{equation}\label{eq:pro1peso}
[w]_{A_p}\leq [w_1]_{A_1}[w_2]_{A_1}^{p-1}.
\end{equation}

Another key feature of the $A_1$ weights that we will use repeatedly is that $(M\mu)^\delta$ is an $A_1$ weight whenever $0<\delta<1$ and $\mu$ is positive Borel measure (this is due to Coifmann and Rochberg, \cite[Theorem 3.4]{GCRdF}). Furthermore we have
\begin{equation*}\label{eq:cr}
[(Mf)^{\delta}]_{A_1}\le \frac{c}{1-\delta},
\end{equation*}
where $c=c_n$.
We will need the following extension of this result for the multilinear maximal operator  $\mathcal{M}$ defined in \eqref{eq:maxmultilinear} which may have its own interest.

\begin{lemma}\label{lem:multi-CR}
Let  $\vec{\mu}$ be a vector of $m$ positive Borel measures on $\mathbb{R}^n$ such that $\mathcal{M}\vec{\mu}(x)<\infty$ for a.e. $x\in\mathbb{R}^n$. Then
\begin{equation}
\left(\mathcal{M}(\vec{\mu})\right)^\delta \in A_1 \quad \mbox{ for any }\, \,\, 0<\delta<\frac{1}{m}
\end{equation}
Moreover,
\begin{equation}
\left[ \left(\mathcal{M}(\vec{\mu})\right)^\delta \right]_{A_1}\le \frac{c}{1-m\delta},
\end{equation}
where $c=c_n$ is some dimensional constant.
\end{lemma}

\begin{proof}
The idea is the same as in the classical Coifman--Rochberg theorem, but using this time the appropriate the weak type boundedness of $\mathcal{M}$:
$$
\mathcal{M}: L^{1}(\mathbb{R}^n) \times  \dots \times L^{1}(\mathbb{R}^n)  \to L^{\frac{1}{m},\infty}(\mathbb{R}^n) .
$$
If \,$w=\left(\mathcal{M}(\vec{\mu})\right)^\delta$, the aim is to prove that, for a given cube $Q$,
$$ \frac{1}{|Q|}\int_Q w(x)\ dx\le \frac{c}{1-m\delta}w(y)\qquad \text{ for all }\ y\in Q.$$
Consider $\tilde{Q}:=3Q$, the dilation of $Q$ and split the vector $\vec{\mu}=\vec{\mu}^0+\vec{\mu}^\infty$ as usual with $\vec{\mu}^0=(\mu^0_1,\dots, \mu^0_m)$ and where $\mu^0_j:=\mu_j\chi_{\tilde{Q}}$ for all $1\le j \le m$. We can handle $\mathcal{M}(\vec{\mu}^\infty)$ as in the $m=1$ case, since in this case the maximal function is essentially constant. For the other part, we have that
\begin{eqnarray*}
\frac{1}{|Q|}\int_Q \left(\mathcal{M}(\vec{\mu}^0)(x)\right)^\delta \ dx & = & \frac{\delta}{|Q|}\int_0^\infty t^{\delta}\left|\left\{x\in Q: \mathcal{M}(\vec{\mu}^0)(x)^\delta >t \right\}\right| \ \frac{dt}{t}\\
 & \le & R^\delta + \frac{\delta}{|Q|}\int_R^\infty t^{\delta}\left|\left\{\mathcal{M}(\vec{\mu}^0)(x)^\delta >t \right\}\right| \ \frac{dt}{t}
\end{eqnarray*}
for any $R>0$ since we trivially have that $\left|\left\{x\in Q: \mathcal{M}(\vec{\mu}^0)(x)^\delta >t \right\}\right|\le |Q|$. Now, we recall
 that $\mathcal{M}$ is a bounded operator from $L^1\times\dots\times L^1 \to L^{1/m,\infty}$. Therefore we can estimate the last integral as
\begin{eqnarray*}
\frac{\delta}{|Q|}\int_R^\infty t^{\delta}\left|\left\{\mathcal{M}(\vec{\mu}^0)(x)^\delta >t \right\}\right| \ \frac{dt}{t}& \le  & 	c \frac{\delta}{|Q|}\int_R^\infty t^{\delta-1-\frac{1}{m}}\ dt \prod_{j=1}^m \|\mu^0_j\|^{1/m}_{L^1}\\
&\le & \frac{c }{1-m\delta}\frac{R^{\delta-\frac{1}{m}}}{|Q|}\prod_{j=1}^m \|\mu^0_j\|^{1/m}_{L^1}
\end{eqnarray*}
for any $\delta<\frac{1}{m}$.  We obtain that
\begin{equation*}
\frac{1}{|Q|}\int_Q \left(\mathcal{M}(\vec{\mu}^0)(x)\right)^\delta \ dx\le R^\delta\left( 1 + \frac{c }{1-m\delta}\frac{\prod_{j=1}^m \|\mu^0_j\|^{1/m}_{L^1}}{R^{\frac{1}{m}}|Q|}\right)
\end{equation*}

Now we choose $R=\displaystyle{\frac{\prod_{j=1}^m \|\mu^0_j\|_{L^1}}{|Q|^m}}$ and we get
\begin{eqnarray*}
\frac{1}{|Q|}\int_Q \left(\mathcal{M}(\vec{\mu})(x)\right)^\delta \ dx & \le &  \left(\frac{c}{1-m\delta}\frac{\prod_{j=1}^m \|\mu^0_j\|_{L^1}}{|Q|^m}\right)^\delta\\
&\le & \frac{c3^{n}}{1-m\delta}\left(\frac{\prod_{j=1}^m \mu_j(\tilde{Q})_j}{|\tilde{Q}|^m}\right)^\delta\\
&\le & \frac{c_{n}}{1-m\delta}\left(\mathcal{M}(\vec{\mu})(x)\right)^\delta
\end{eqnarray*}

\end{proof}

The following Proposition can be viewed as an integral version of the main result of the previous section, namely Theorem \ref{thm:Feff-Steinish}. It follows from Lerner's formula as well, but this \emph{integral} version is the key to obtain the result for commutators, which cannot be obtained by means of the first approach.  

\begin{proposition}\label{pro:previo}
Let $f$ be a measure function such that $\supp f\subset Q$, being $Q$ a fixed cube. Let $0<\delta<1$ and let $w\in A_q$. Then we have that
\begin{equation}\label{eq:previo}
\|f-m_f(Q)\|_{L^1(w,Q)}\leq c\, 2^q\, [w]_{A_q}\|M^{\#,d}_{\delta}(f)\|_{L^1(w,Q)}
\end{equation}
\end{proposition}
\begin{proof}
We start with a pointwise estimate, which follows from Lerner's formula, taking into account the  definition  of the oscillation and \eqref{eq:rearravg-delta}.
$$
|f(x)-m_f(Q)|\leq c\, M_\delta^\#f(x)+c\,\sum_{k,j}\inf_{Q_j^k}M_\delta^\#(f)\,\chi_{Q_j^k}(x).
$$
Then, taking norms,
\begin{equation*}
\norm{f-m_f(Q)}_{L^1(w,Q)}\le  c\|M^{\#,d}_{\delta}(f)\|_{L^1(w,Q)}+c\sum_{k,j}\int_{Q_j^k}\inf_{Q_j^k} M_\delta^\#(f)w(x)dx.
\end{equation*}

Now we recall that the family $\{E_j^k\}$  satisfies \eqref{eq:remark-Lerner} and  use the following property  of the $A_q$ class of weights: let  $w\in A_q$ and let $Q$  be a cube, then for each measurable sets such that $E\subset Q$,
\begin{equation*}
 w(Q)\le\left(\frac{|Q|}{|E|}\right)^q[w]_{A_q}w(E).
\end{equation*}
Since for any index $(j,k)$ we have the property $|Q_j^k|\le2|E_j^k|$, it follows that
\begin{equation*}
 w(Q_j^k)\le2^q[w]_{A_q}w(E_j^k).
\end{equation*}
If we apply this on each term of the sum, we obtain
\begin{equation*}
 \int_{Q_j^k}\inf_{Q_j^k} M_\delta^\#(f)w(x)dx \le c 2^q\,[w]_{A_q}\inf_{Q_j^k}M_\delta^\#(f)\, w(E_j^k)
\end{equation*}
Finally, we obtain that
\begin{equation*}
\norm{f-m_f(Q)}_{L^1(w,Q)} \leq c \,2^q\,[w]_{A_q}\norm{M_\delta^\#(f)}_{L^1(w,Q)},
\end{equation*}
since $\{E_j^k\}$ is a pairwise disjoint subsets family.

\end{proof}

The following lemma gives a way to produce $A_{1}$ weights with special control on the constant. It is based on the  so called Rubio de Francia iteration scheme or algorithm.

\begin{lemma}\cite{LOP3}\label{lem:rubiodefrancia}
Let $M$ be the usual Hardy--Littlewood maximal operator and let $0<r<\infty$. Define the operator $R:L^r(\mathbb{R}^n)\to L^r(\mathbb{R}^n)$ as follows. For a given $h\in L^{r}(\mathbb{R}^n)$, consider the sum:
$$
R(h)=\sum_{k=0}^{\infty}{\frac1{2^k}}{\frac{M^{k}h}{\norm{M}_{L^r(\mathbb{R}^n)}^{k}}},
$$
Then $R$ satisfies the following properties:
\begin{enumerate}
  \item[(i)] $h\leq R(h)$;
  \item[(ii)] $\norm{Rh}_{L^r(\mathbb{R}^n)}\leq 2\norm{h}_{L^r(\mathbb{R}^n)}$;
  \item[(iii)] For any nonnegative $h\in L^r(\mathbb{R}^n)$, we have that $Rh \in A_1$with
$$[Rh]_{A_1} \leq 2\norm{M}_{L^r(\mathbb{R}^n)} \leq c_n\,r'.$$
\end{enumerate}
\end{lemma}

\subsection{The model case}

Consider two nonnegative operators $T_1$ and $T_2$, where typically $T_1$ is the absolute value of a singular operator and $T_2$ is an appropriate maximal operator that will act as a control operator.  As in the introduction we will be slightly vague on the use of the notation, since here $f$ will stand for a single function, a vector or an infinite sequence of functions, depending on the operators. Assume that, for any cube $Q$, we have a weighted $L^1$  local Coifman-Fefferman type inequality.  To be more precise we will assume the following:

1)  There is an special positive parameter  $\beta$ and an index $1\le q\le \infty$ for which we can find a constant $c$ such that for any $w\in A_q$  and any cube $Q$,
\begin{equation}\label{eq:genericC-F}
 \norm{T_1 f}_{L^1(w,Q)}\le c [w]^\beta_{A_q}\norm{T_2 f}_{L^1(w,Q)},\qquad 
\end{equation}
for appropriate functions $f$. The parameter $\beta$ is key in the sequel.

2)  Suppose that the (maximal type) operator $T_2$ is so that $(T_2f)^\frac{1}{q-1}\in A_1$ with
\begin{equation}\label{eq:maximalA1}
[(T_2f)^\frac{1}{q-1}]^{q-1}_{A_1}\le a
\end{equation}
where $a$ is a constant independent of $f$.

The general purpose is to estimate the level set function $\varphi$ as in \eqref{eq:generic} 
$$\varphi(t):=\frac{1}{|Q|} |\{x\in Q: |T_1f(x)|> t|T_2f(x)|\}|.
$$ 

We start by applying Chebychev's inequlity for some $p>1$ that will be chosen later:
\begin{eqnarray*}
|\{x\in Q: |T_1(f)(x)|> t\,|T_2(f)(x)|\}|&\leq& {\frac1{t^p}}\int_{Q} \left|{\frac{T_1f(x)}{T_2f(x)}}\right|^p\, dx\\
&=&{\frac1{t^p}}\norm{\displaystyle{\frac{T_1f(x)}{T_2f(x)}}}_{L^p(Q)}^p\\
&\le&{\frac1{t^p}}\left(\int_{Q}  {\frac{T_1f(x)}{T_2f(x)}} \,h(x)\,dx\right)^p
\end{eqnarray*}
for some nonnegative $h$ such that  $\norm{h}_{L^{p'}(Q)}=1$. Now we apply Rubio de Francia's algorithm (Lemma \ref{lem:rubiodefrancia}) and the key hypothesis \eqref{eq:genericC-F} to obtain that
\begin{eqnarray*}
\int_{Q} {\frac{T_1(f)(x)}{T_2(f)(x)}} \,h(x)\,\,dx &\leq&\int_{Q} T_1(f)(x)\,T_2(f)(x)^{-1}R(h)(x)dx\\
&\leq& c\,[R(h)(T_2f)^{-1}]^\beta_{A_q}\int_{Q}T_2(f)(x)\frac{R(h)(x)}{T_2(f)(x)}\, dx\\
&=& c\, [R(h)(T_2f)^{-1}]^\beta_{A_q}\int_{Q} R(h)(x)\, dx\\
&\leq&c\, [R(h)(T_2f)^{-1}]^\beta_{A_q}2\,\norm{h}_{L^{p'}(Q)}|Q|^{1/p}\\
&=&2c\, [R(h)(T_2f)^{-1}]^\beta_{A_q}\,|Q|^{1/p}.
\end{eqnarray*}
Since $R(h)\in A_1$ we can use formula \eqref{eq:pro1peso}
$$[R(h)(T_2f)^{-1}]_{A_q} \le [R(h)]_{A_1}[(T_2f)^\frac{1}{q-1}]^{q-1}_{A_1} \leq p\,a.$$
since by Lemma \ref{lem:rubiodefrancia}, (iii) \,$[R(h)]_{A_1}\le p$\,  and the constant in \eqref{eq:maximalA1} is uniform on $f$.

Then, if we choose $p$ such that $\displaystyle{e^{-1}=\frac{ (ap)^\beta}{t}}$, we get
\begin{eqnarray*}
|\{x\in Q: T_1f(x)> t\, T_2f(x)\}|&\leq& 2\,c\,\left({\frac{(a\,p)^\beta}{t}}\right)^p |Q|\\
& \le & 2c\, e^{-\alpha t^\frac{1}{\beta}}|Q|
\end{eqnarray*}
where $\alpha$ depends on $\beta$ and $q$ and hence $\varphi(t)\leq 2 e^{-\alpha t^\frac{1}{\beta}}$.

Note that this model example reveals that the two hypothesis that we need to fulfill:

\begin{enumerate}
 \item[(H1)] A local Coifman--Fefferman inequality like \eqref{eq:genericC-F} with the sharpest exponent $\beta$ on the constant of the weight which controls the decay rate of the level set function $\varphi(t)$.
\item[(H2)] An appropriate power of the maximal operator $T_2$ should be a $A_1$ weight. This is the case in all the operators we consider in this paper and follows essentially from a suitable variations of Coifman--Rochberg's theorem (Lemma \ref{lem:multi-CR}).
\end{enumerate}

This scheme will be followed in the proof of the main results. In each case, we will show how to derive the appropriate local Coifman--Fefferman inequality with the correct exponent and will check that the $A_1$ constant of the control operator is uniformly bounded.

\subsection{Proofs for C--Z operators, vector valued extensions and  multilinear C--Z operators - Second approach}

We start by proving local C--F inequalities. The central tool is Proposition \ref{pro:previo}.

\begin{theorem}\label{thm:CFlineal}
Let $w\in A_{q}$, with $1\leq q<\infty$.
\begin{enumerate}
 \item Let $T$ be a Calder\'on--Zygmund integral operator. Let $f$ be a function such that $\supp{f}\subseteq Q$. Then, there exists a constant $c=c_{n}$ such that
\begin{equation*}
\norm{T^*f}_{L^1(w,Q)}\leq c\,2^q\,[w]_{A_q}\norm{Mf}_{L^1(w,Q)}.
\end{equation*}
\item Let $\overline{T}_{q}$ the vector--valued extension of Calder\'on--Zygmund integral operators. Let $f=\{f_j\}_{j=1}^{\infty}$ be a vector-valued function such that $\supp{f}\subseteq Q$. Then, there exists a constant $c=c_{n}$ such that
\begin{equation*}
\norm{\overline{T}_{q}f}_{L^1(w,Q)}\leq c\,2^q\,[w]_{A_q}\norm{M(|f|_q)}_{L^1(w,Q)}.
\end{equation*}
\item Let $T$ be a $m$-linear C--Z operator. Let $\vec{f}\,$ is a vector of $m$ functions such that $\supp{f_j}\subseteq Q$. Then, there exists a constant $c=c_{n}$ such that
\begin{equation*}
\norm{T(\vec{f}\,)}_{L^1(w,Q)}\leq c\,2^q\,[w]_{A_q}\norm{\mathcal{M}(\vec{f}\,)}_{L^1(w,Q)}.
\end{equation*}
\end{enumerate}
\end{theorem}

\begin{proof}
In all three cases we start with Proposition \ref{pro:previo} and use \eqref{eq:MsharpT-vs-Mf}, \eqref{eq:vectorialmaximalsharp} and \eqref{eq:multilinmaximalsharp} to control the sharp maximal function. It remains to prove that we can control the median value in each case. For $T^*$, we already have done it in \eqref{eq:medianT-vs-M}. There we proved that
\begin{equation*}
m_{T^*f}(Q) \leq {\frac{c}{|Q|}}\int_{Q}|f(x)|\,dx\le cMf(y),\quad \text{ for all }\ y\in Q,
\end{equation*}
and hence
$$w(Q)\,m_{T^*f}(Q)\le c\,\norm{Mf}_{L^1(w,Q)}.$$
The case of the vector-valued operators follows the same steps. The multilinear case also follows from Kolmogorov's inequality and the weak type boundedness of multilinear C--Z operators.

\end{proof}
At this point, we have proved Coifman--Fefferman inequalities like \eqref{eq:genericC-F} with  $\beta=1$ in all the cases. It remains to check that the factorizazion argument can be performed. For C--Z operators and its vector-valued extension, we can use \eqref{eq:genericC-F} with $q=3$ (or any larger $q$). Therefore, \eqref{eq:maximalA1} holds in both cases by Lemma \ref{lem:multi-CR} for $m=1$. Note that what we have to control in both cases is $ [M(\mu)^\frac{1}{2}]^2_{A_1}$. For the multilinear case, we have that, for $q=m+2$,

\begin{equation*}
 [\mathcal{M}(\vec{f}\,)^{\frac{1}{q-1}}]^{q-1}_{A_1}\le \left(\frac{C_n}{1-\frac{m}{m+1}}\right)^{m+1}.
\end{equation*}
So we finish the proof of Theorem \ref{thm:singinteg} and Theorem \ref{thm:multisinginteg}.

\subsection{Results for commutators.}
As before, we need to prove an appropriate Coifman--Fefferman inequality.
\begin{theorem}\label{thm:cflocalcom}
Let $w\in A_{q}$ be, with $1\leq q<\infty$. Let $T$ be a C--Z operator and $b\in BMO$. Let $f$ be a function such that $\supp{f}\subseteq Q$. Then, there exists a dimensional constant $c=c_{n}$ such that
\begin{equation}\label{eq:cflocal-commu}
\norm{[b,T]f}_{L^1(w,Q)}\leq c\,\norm{b}_{\strt{1.7ex}BMO}2^{2q}[w]_{A_q}^2\norm{M^2f}_{L^1(w,Q)}.
\end{equation}
In the higher order commutator case we get that
\begin{equation}\label{eq:cflocal-commu-super}
\norm{T_{b}^kf}_{L^1(w,Q)}\leq c\,\norm{b}_{\strt{1.7ex}BMO}^k2^{(k+1)q}[w]_{A_q}^{k+1}\norm{M^{k+1}f}_{L^1(w,Q)}.
\end{equation}
\end{theorem}
\begin{proof}

Using Proposition \ref{pro:previo} we get that
\begin{eqnarray*}
\norm{T_bf}_{L^1(w,Q)} & \leq & c2^q [w]_{A_q}\int_{Q}M^{\#,Q}_{\delta}(T_bf)\, w(x)dx+w(Q)m_{[b,T]f}(Q)\\
&= & I + II
\end{eqnarray*}
For the first term, by \eqref{eq:mdeltaepsilon}  we have that
\begin{equation*}
I  \leq  c 2^q[w]_{A_q}\|b\|_{BMO}\left(\norm{M^Q_{\varepsilon}(Tf)}_{L^1(w,Q)} +\norm{M^2f}_{L^1(w,Q)}\right)
\end{equation*}
Now we write $L(Q):=w(Q)m_{M^Q_{\varepsilon}(Tf)}(Q) $ and  we apply Proposition \ref{pro:previo}  with some $0<\delta<\varepsilon$ to the first norm to obtain that
\begin{eqnarray*}
\norm{M^Q_{\varepsilon}(Tf)}_{L^1(w,Q)}  & \le & c2^q[w]_{A_q}\norm{M^{\#,Q}_\delta\left(M^Q_{\varepsilon}(Tf)\right)}_{L^1(w,Q)}+ L(Q) \\
 &\leq &c2^q[w]_{A_q}\norm{M^{\#,Q}_\varepsilon(Tf)}_{L^1(w,Q)}+ L(Q)\\
&\leq &c2^q[w]_{A_q}\norm{Mf}_{L^1(w,Q)}+ L(Q)
\end{eqnarray*}
by \eqref{eq:dosmaximales} and \eqref{eq:MsharpT-vs-Mf}. Now we have to bound $L(Q)$. We apply property \eqref{eq:medianVSnormdelta} for the median value with some $0<\delta<\varepsilon$, Kolmogorov's inequality twice and the weak type of both $M$ and $T$:
\begin{eqnarray*}
m_{M_{\varepsilon}(\chi_QTf)}(Q) & \le & c\left(\frac{1}{|Q|}\int_Q M^Q_{\varepsilon}(Tf)^\delta\ dx\right)^{\frac{1}{\delta}} \\
& \le & \left(\frac{1}{|Q|}\int_Q M^Q(|Tf|^\varepsilon)^\frac{\delta}{\varepsilon}\ dx\right)^{\frac{\varepsilon}{\delta}\frac{1}{\varepsilon}} \\
& \le & \norm{M^Q(|Tf|^\varepsilon)}^\frac{1}{\varepsilon}_{L^{1,\infty}(\frac{dx}{|Q|},Q)}\\
& \le & \left(\frac{1}{|Q|}\int_Q |Tf|^\varepsilon\ dx\right)^\frac{1}{\varepsilon}\le Mf(x) \qquad \text{ for all } x\in Q.\\
\end{eqnarray*}
Therefore, we have that $w(Q)m_{M_{\varepsilon}(Tf)}(Q) \le \norm{Mf}_{L^1(w,Q)}$. Combining all previous estimates, we get
$$
I  \leq  c 2^{2q}[w]^2_{A_q}\|b\|_{BMO}\norm{Mf}_{L^1(w,Q)} +c [w]_{A_q}\|b\|_{BMO}\norm{M^2f}_{L^1(w,Q)}
$$
From this estimate, using that $[w]_{A_{q}}\ge 1$ and by dominating $M$ by $M^2$, we obtain the desired result if we can control $II$, involving the median value of the commutator. To that end, we will use the weak estimate from Theorem \ref{thm:weakcomm}. For $\phi(t)=t(1+\log^+(t))$,
\begin{eqnarray*}
 m_{[b,T]f}(Q)& \le & \left(\frac{1}{|Q|}\int_Q |T_bf|^\delta\ dx\right)^\frac{1}{\delta}\\
& \le &  \left(\frac{1}{|Q|}\int_0^\infty \delta t^{\delta-1} |\{x\in Q: |T_bf|>t\}|\ dt\right)^\frac{1}{\delta}\\
& \le &  \left(R^\delta + \frac{1}{|Q|}\int_R^\infty \delta t^{\delta-1} \int_Q \phi\left(\frac{|f(x)|}{t}\right)\ dx dt\right)^\frac{1}{\delta}\\
\end{eqnarray*}
for any $R>0$ (to be chosen). By the submultiplicativity of $\phi$, we have that
\begin{eqnarray*}
 m_{[b,T]f}(Q)& \le & \left(R^\delta + \frac{1}{|Q|}\int_Q \phi\left(|f(x)|\right)\ dx \int_R^\infty\delta t^{\delta-1}\phi(1/t)\ dt \right)^\frac{1}{\delta}\\
& \le & \left(R^\delta + \frac{R^{\delta-1} }{|Q|}\int_Q |f(x)|(1+\log^+   \frac{|f(x)|}{|f|_Q}   )\ dx \right)^\frac{1}{\delta}\\
&\le& R \left(1+ \frac{1}{R|Q|}\int_Q Mf(x)\ dx \right)^\frac{1}{\delta}.
\end{eqnarray*}
where we have used the following well known estimate essentially due to E. Stein \cite{Stein-LlogL},
\begin{equation*}
\int_Q w\, \log(e +
\frac{w}{w_Q})\, dx  \le c_n\,
\int_{Q} M(w\chi_Q) \, dx \qquad w\geq 0.
\end{equation*}
If we now choose $R=\frac{1}{|Q|}\int_Q Mf(x)\ dx$, then we obtain that
\begin{equation*}
 m_{[b,T]f}(Q)\le M^2f(x)\qquad \text{ for all } x\in Q.
\end{equation*}
This clearly implies that $w(Q)m_{[b,T]f}(Q)\le\norm{M^2f}_{L^1(w,Q)}$ and the proof of inequality \eqref{eq:cflocal-commu} is complete. The higher order commutator bound \eqref{eq:cflocal-commu-super} is technically more complicated, but it follows from the same ideas by using an induction argument. The details can be found in \cite{O2}

\end{proof}

Hence, we have the hypothesis (H1) of our model. We finish the proof of Theorem  \ref{thm:commutator} by proving that we have also the second hypothesis (H2). Namely, we have $M^2$ acting as a control operator, so we have to prove that, for some $q>1$,
\[
[(M^2f)^\frac{1}{q-1}]^{2(q-1)}_{A_1}\le C.
\]
But since $M^2(f)=M(M(f))$, this is, once again, Coifman-Rochberg theorem. We only have to pick, for instance, $q=3$.

\subsection{Proof for the square functions and for the vector-valued maximal function - Second approach}

For the dyadic square function $S$, the statement of Theorem \ref{thm:squarefunction} and the discussion about the ``model'' of proof suggest that we need a C-F inequality with $\beta=\frac{1}{2}$ as exponent on the weight. This is essentially what we borrow from \cite{CMP-ADV} in  \eqref{eq:squareOscilaDyadic} for the dyadic case and from \eqref{eq:sharplocal-SquareCont} for the continuous case.
From those estimates for the oscillation and using Lerner's formula, we can derive the following Coifman-Fefferman type inequality.
\begin{lemma}\label{lem:squareCF}
Let $S$ be the dyadic square function operator $S_d$ or the continuous square function $g_{\mu}^{*}$  and let $w\in A_q$. Then for any function $f$ and every cube $Q$
\begin{equation}\label{eq:squareCF}
\int_Q(Sf(x))^2 w(x)\ dx \le c\,2^q\, [w]_{A_q}\int_Q(Mf(x))^2 w(x)\ dx
\end{equation}
\end{lemma}
\begin{proof}
It follows directly from Lerner's formula in both cases. For the median value, we can use Kolmogorov and the weak $(1,1)$ type of the operator as in the previous cases.

\end{proof}
Now we prove Theorem \ref{thm:squarefunction} by using our model, but with a twist.
\begin{proof}[of Theorem \ref{thm:squarefunction}]
Motivated by Lemma \ref{lem:squareCF}, we write the level set function $\varphi(t)$ as follows.
\begin{eqnarray*}
\left|\left\{x\in Q: \frac{Sf(x)^2}{Mf(x)^2}>t^2\right\}\right| &\le & \frac{1}{t^{2p}}\int_{Q} {\frac{Sf(x)^{2p}}{Mf(x)|^{2p}}}\, dx\\
&=& \frac{1}{t^{2p}}\norm{\displaystyle{\left({\frac{Sf}{Mf}}\right)^2}}_{L^p(Q)}^p\\
&\le&\frac1{t^{2p}}\left(\int_{Q}\frac{Sf(x)^2}{Mf(x)^2}\,h(x)\,dx\right)^p
\end{eqnarray*}
for some $h$ such that $\norm{h}_{L^{p'}(Q)}=1$. Now we apply Rubio de Francia's algorithm and \eqref{eq:squareCF} to obtain that
\begin{eqnarray*}
\int_{Q}\frac{Sf(x)^2}{Mf(x)^2}\,h(x)\,dx & \leq & \int_{Q} Sf(x)^2Mf(x)^{-2}\,R(h)(x)\,dx \\
&\leq& [R(h)\,(Mf)^{-2}]_{A_q}\int_{Q}R(h)(x)\, dx\\
&\leq& [R(h)\,(Mf)^{-2}]_{A_q}2\,\norm{h}_{L^{p'}(Q)}|Q|^{1/p}.
\end{eqnarray*}
The same factorization argument yields
$$[Rh(Mf)^{-2}]_{A_q} \le [(Mf)^\frac{2}{q-1}]^{(q-1)}_{A_1}p\le c\,p$$
for $q=5$ by Lemma \ref{lem:multi-CR}.

We now choose $\displaystyle{e^{-1}=\frac{c_n p}{t^2}}$ to finally obtain that
\begin{equation*}
|\{x\in Q: Sf(x)> tMf(x)\}|\leq \left({\frac{c_n p}{t^2}}\right)^p |Q| \le  e^{-\alpha t^2}|Q|
\end{equation*}

\end{proof}

Finally, the proof for the vector valued extension of the maximal function follows the same steps.
\begin{proof}[of Theorem \ref{thm:vectorvaluedMaximal}]
The proof can be carried by replacing the ``2'' by ``$q$'' in case of the square function. The key estimate for the oscillation is inequality \eqref{eq:vectormaximalOscila}.

\end{proof}

\section{Acknowledgement}
The second author is supported by the Spanish Ministry of Science and Innovation grant MTM2009-08934,
the second and third authors are also supported by the Junta de Andaluc\'ia, grant FQM-4745.

%
%
%
%

\end{document}